\documentclass{amsart}
\usepackage{amsmath}
\usepackage{amsthm}
\usepackage{listings}
\usepackage{xcolor}
\usepackage
[final]
{showlabels}
\usepackage{amssymb}
\usepackage{tikz-cd}
\usepackage{algorithm2e}
\usepackage{enumitem}  
\usepackage{stmaryrd}

\usetikzlibrary{decorations.markings}
\tikzset{negated/.style={
        decoration={markings,
            mark= at position 0.5 with {
                \node[transform shape] (tempnode) {$\backslash$};
            }
        },
        postaction={decorate}
    }
}

\newtheorem{theorem}{Theorem}[section]
\newtheorem{proposition}[theorem]{Proposition}
\newtheorem{lemma}[theorem]{Lemma}
\newtheorem{corollary}[theorem]{Corollary}

\newtheorem{alphatheorem}{Theorem}

\theoremstyle{definition}

\newtheorem{remark}[theorem]{Remark}
\newtheorem{conjecture}[theorem]{Conjecture}

\setenumerate{label={\normalfont(\textit{\roman*})}, ref={\textrm{\roman*}},
wide, leftmargin=*, labelwidth=!, labelindent=0pt, 
}

\usepackage{ifthen}
\usepackage{scalerel}
\usepackage{hyperref}
\usepackage{mathrsfs}
\usepackage[mathscr]{euscript}
\usepackage{lineno}

\newcommand{\linspan}{\operatorname{span}}

\newcommand{\floor}[1]{\left\lfloor #1 \right\rfloor}
\newcommand{\bra}[1]{\left( #1 \right)}

\renewcommand{\tilde}{\widetilde}
\renewcommand{\bar}{\overline}

\newcommand{\abs}[1]{\left|#1\right|}
\newcommand{\set}[2]{\left\{ #1 \ \middle| \ #2 \right\} }

\newcommand{\NN}{\mathbb{N}}

\newcommand{\QQ}{\mathbb{Q}}
\newcommand{\ZZ}{\mathbb{Z}}
\newcommand{\RR}{\mathbb{R}}
\newcommand{\CC}{\mathbb{C}}

\newcommand{\cN}{\mathcal{N}}

\newcommand{\cM}{\mathcal{M}}

\newcommand{\mynote}[2]{\noindent
{\bfseries\sffamily\scriptsize#1}
{\small$\blacktriangleright$\textsf{\textsl{#2}}$\blacktriangleleft$}}
\newcommand\JK[1]{\mynote{JK}{{\color[rgb]{0.5,0.2,0.0} #1}}}

\definecolor{new}{RGB}{0,100,0}
\definecolor{old}{RGB}{220,220,220}
\definecolor{newnew}{RGB}{0,0,128}
\definecolor{good}{RGB}{122,111,0}
\definecolor{newgood}{RGB}{211, 75, 8}
\definecolor{final}{RGB}{128,0,0}

\renewcommand{\color}[1]{\texttt{[#1]}}

\renewcommand{\color}[1]{}
\renewcommand{\JK}[1]{}

\newcommand{\ind}{\mathbf{1}}
\newcommand{\sinv}{dilation-invariant}
\newcommand{\Sinv}{Dilation-invariant}
\newcommand{\pc}{pattern}
\newcommand{\Pc}{Pattern}

\renewcommand{\subset}{\subseteq}

\newcommand*\patchAmsMathEnvironmentForLineno[1]{\expandafter\let\csname old#1\expandafter\endcsname\csname #1\endcsname
  \expandafter\let\csname oldend#1\expandafter\endcsname\csname end#1\endcsname
  \renewenvironment{#1}{\linenomath\csname old#1\endcsname}{\csname oldend#1\endcsname\endlinenomath}}\newcommand*\patchBothAmsMathEnvironmentsForLineno[1]{\patchAmsMathEnvironmentForLineno{#1}\patchAmsMathEnvironmentForLineno{#1*}}\AtBeginDocument{\patchBothAmsMathEnvironmentsForLineno{equation}\patchBothAmsMathEnvironmentsForLineno{align}\patchBothAmsMathEnvironmentsForLineno{flalign}\patchBothAmsMathEnvironmentsForLineno{alignat}\patchBothAmsMathEnvironmentsForLineno{gather}\patchBothAmsMathEnvironmentsForLineno{multline}}

\begin{document}

\author[J.\ Konieczny]{Jakub Konieczny}
\address[J.\ Konieczny]{Camille Jordan Institute, 
Claude Bernard University Lyon 1,
43 Boulevard du 11 novembre 1918,
69622 Villeurbanne Cedex, France}
\address{Faculty of Mathematics and Computer Science, Jagiellonian University in Krak\'{o}w, \L{}ojasiewicza 6, 30-348 Krak\'{o}w, Poland\newline}
\email{jakub.konieczny@gmail.com}

\title[Noncorrelated pattern sequences]{Algorithmic classification of \\ noncorrelated binary pattern sequences}

\begin{abstract}\color{final}
	The main subject of this paper are binary \pc{} sequences, that is, sequences of the form $(-1)^{\#(n,A)}$ where $A$ is a set of strings of $\mathtt 0$s and $\mathtt 1$s, and $\#(n,A)$ denotes the total number of times patterns from $A$ appear in the binary expansion of $n$. A sequence is said to be noncorrelated if the corresponding spectral measure is equal to the Lebesgue measure.
	
	We show that it is possible to algorithmically verify if a given binary \pc{} sequence is noncorrelated. As an application, we compute that there are exactly $2272$  noncorrelated binary \pc{} sequences of length $\leq 4$. If we restrict our attention to patterns that do not end with $\mathtt{0}$, we put forward a sufficient condition for a \pc{} sequence to be noncorrelated. We conjecture that this condition is also necessary, and verify this conjecture for lengths $\leq 5$.
\end{abstract}


\maketitle 

\newcommand{\M}[1]{\mathrm{M}_{#1}}
\newcommand{\Mlog}[1]{\mathrm{M}_{#1}^{\log}}
\newcommand{\gammalog}{\gamma^{\log}}

\section{Introduction}\label{sec:Intro}

\color{newgood}
Uniformity properties of sequences defined in terms of digital expansions have long been studied. Consider, for instance, the Thue--Morse sequence $t(n) = (-1)^{s_2(n)}$, where $s_2(n)$ denotes the sum of binary digits of $n$, discussed at length by Allouche and Shallit in the survey paper \cite{AlloucheShallit-1999}. It was shown by Gelfond \cite{Gelfond-1967} (see also \cite{MauduitSarkozy-1998}) that $t(n)$ is equidistributed in arithmetic progressions:
\begin{equation}\label{eq:870:1}
	\lim_{N \to \infty} \abs{\set{ 0 \leq n < N}{ t(An+B) = +1 }}/N = 1/2
\end{equation}
for all $A \in \NN$ and $B \in \NN_0$, and the rate of convergence can be made explicit. Analogous results hold also for other bases, with mild additional assumptions to account for the fact that $s_{k}(n) \equiv n \bmod{k-1}$. Mauduit and S{\'a}rk{\"o}zy \cite{MauduitSarkozy-1998} also observed that the Thue--Morse sequence admits large self-correlations. Here, the (self-)correlation coefficients of a sequence $a \colon \NN \to \CC$  are defined by
\begin{equation}\label{def:correlation}
\gamma_{a}(m) := \lim_{N \to \infty} \frac{1}{N} \sum_{n=0}^{N-1} a(n)\bar a(n+m),
\end{equation}
and a simple computation shows that $\gamma_{t}(1) = -1/3 \neq 0$ (see Section \ref{sec:Correlation} for details). By the same token,
\begin{equation}\label{eq:870:2}
	\gamma_t(2^{\ell}) = -1/3 \text{ for all } \ell \in \NN_0,
\end{equation}
meaning in particular that $\gamma(m) \not\to 0$ as $m \to \infty$. On the other hand, the coefficients $\gamma_t(m)$ tend to be rather small; in particular 
\begin{equation}\label{eq:870:21}
	\lim_{N\to\infty} \frac{1}{N} \sum_{m=0}^{N-1} \gamma_t(m)^2 = 0,
\end{equation}
which follows e.g.{} from results in \cite{Coquet-1976}. The spectral measure $\mu_a$ on $\RR/\ZZ$ associated to a sequence $a \colon \NN \to \CC$ is characterised by the identity $\int_{\RR/\ZZ} \exp(2\pi i m t) d\mu_a(t) = \gamma_a(m)$, and \eqref{eq:870:21} is equivalent to absolute continuity of $\mu_t$.

Many other notions of uniformity have been investigated for the Thue--Morse sequence. In an influential paper, Mauduit and Rivat showed that $t(n)$ and its analogues in different bases are equidistributed along the primes \cite{MauduitRivat-2010}. Drmota, Mauduit and Rivat \cite{DrmotaMauduitRivat-2019} showed that $t(n^2)$ is a normal sequence, meaning that each finite sequence of $\pm 1$s appears with the expected frequency. 
Spiegelhofer \cite{Spiegelhofer-2018} proved that $t(n)$ has level of distribution $1$, which is a far-reaching quantitative generalisation of \eqref{eq:870:1} and can be used to show equdistribution along Piatetski--Shapiro sequences $\floor{n^c}$, $1 < c < 2$ (see \cite{FouvryMauduit-1996} for analogous, but somewhat weaker, results in different bases).   
It was also shown by the author \cite{Konieczny-2019} that $t(n)$ has small Gowers norms, meaning that it is uniform from the point of view of higher order Fourier analysis. 

\color{newgood}
Another oft-studied sequence carries the name of Rudin--Shapiro and is given by $r(n) = (-1)^{\#(\mathtt{11},n)}$, where $\#(\mathtt{11},n)$ denotes the number of times the pattern $\mathtt{11}$ appears in the binary expansion of $n$, allowing overlaps. Similarly to the Thue--Morse sequence, the Rudin--Shapiro sequence is equidistributed in arithmetic progressions and along the primes \cite{MauduitRivat-2015}, and has small Gowers norms \cite{Konieczny-2019}. However, in contrast to \eqref{eq:870:2}, the Rudin--Shapiro sequence is \emph{noncorrelated}, by which we mean that $\gamma_r(m) = 0$ for all $m \geq 1$ or, equivalently, that the spectral measure $\mu_r$ is is the Lebesgue measure. Intuitively, noncorrelated sequences are free of any sort of periodic behaviour.	

The Thue--Morse and the Rudin--Shapiro sequences are special cases of what we call binary \pc{} sequences. In general, a binary \pc{} sequence takes the form 
\[
	a(n) = a_A(n) = (-1)^{\#(A,n)},
\]
where $A$ is a finite set of patterns over the alphabet $\{\mathtt 0, \mathtt 1\}$ and $\#(A,n)$ denotes the total number of appearances of patterns from $A$ in the binary expansion of $n$ (see Section \ref{sec:Background} for details). \Pc{} sequences were studied in a more general context by Morton and Mourant \cite{MortonMourant-1989,Morton-1990}, Coquet, Kamae and Mend\`{e}s France \cite{CoquetKamaeMendesFrance-1977}, and Boyd, Cook and Morton \cite{BoydCookMorton-1989}. Generalised Rudin--Shapiro sequences and their correlation coefficients were studied by Allouche and Liardet \cite{AlloucheLiardet-1991}. Finally, Zheng, Peng and Kamae \cite{ZhengPengKamae-2018} studied correlation coefficients of binary \pc{} sequences, and obtained a complete classification of noncorrelated sequences corresponding to sets of patterns of length $\leq 3$. Examples of sets $A$ that give rise to noncorrelated sequences include:
\begin{itemize}
\item $\{ \mathtt{11} \}$ (then $a_A(n) = r(n)$ is the Rudin--Shapiro sequence);
\item $\{ \mathtt{11}, \mathtt{1} \}$ (then $a_A(n) = r(n) t(n)$) and $\{ \mathtt{10}, \mathtt{1} \}$ (then $a_A(n) = (-1)^n r(n)$);
\item $\{ \mathtt{101}, \mathtt{111} \}$, or more generally $\{ \mathtt{101}, \mathtt{111} \} \cup B$ for a set $B \subset \{\mathtt{0},\mathtt{1}\}^2$.
\end{itemize}

In this paper, we extend the result of \cite{ZhengPengKamae-2018} to patterns of length $\leq 4$ and put the findings in a wider context provided by the theory of automatic and regular sequences. Many of the ideas we use have their analogues and prototypes in \cite{ZhengPengKamae-2018}; throughout the paper we give references to the relevant results therein.

Unfortunately, there does not appear to be a simple criterion that determines if a given \pc{} sequence is noncorrelated (except for the partial information suggested by Conjecture \ref{conj:NC->inv_x_per} below). Due to practical limitations we only state a counting result here, as opposed to a complete list.\footnote{The list, together with the code which can be used to produce it, is available from the author.}
\begin{alphatheorem}\label{thm:A}
	There are precisely $2272$ noncorrelated binary \pc{} sequences corresponding to patterns of length $\leq 4$. 
\end{alphatheorem}

As a key step towards obtaining the above result, we reduce the task of verifying whether a given binary \pc{} sequence is noncorrelated to a finite computation, which can then be automated. The time complexity of the resulting algorithm is polynomial in $2^\ell$, where $\ell$ denotes the length of patterns under consideration. Since it takes approximately $2^{\ell}$ bits to specify a binary \pc{} sequence, this is optimal up to improvements in the exponent. 

\begin{alphatheorem}\label{thm:B}
	There exists an algorithm which, given a finite set of patterns $A \subset \{\mathtt 0, \mathtt 1\}^\ell$,  performs $2^{O(\ell)}$ operations and decides if the corresponding \pc{} sequence $a_A$ is noncorrelated.
\end{alphatheorem}

While we keep the exposition fairly self-contained, we also wish to emphasize that the above problem can be seen as a part of a larger theory. We note that binary \pc{} sequences are $2$-automatic (see Section \ref{sec:Background} for the relevant definitions). A crucial component of our reasoning is Theorem \ref{cor:gamma-is-reg}, which assets that the correlation sequences coming from automatic sequences are regular. While this result will not come as a surprise to the experts in the field, to the best of our knowledge it does not appear in print elsewhere. Its importance stems from the fact that a regular sequence admits a simple recursive description, an hence many properties are easily verified for such a sequence. In our particular application, we reduce the task of determining if a \pc{} sequence is noncorrelated to the ostensibly simpler task of determining if a $2$-regular sequence is identically zero.

The problem of classifying noncorrelated \pc{} sequences becomes more tractable if we impose additional assumptions on the set of patterns under consideration. Let us call a binary \pc{} sequence $a \colon \NN_0 \to \{+1,-1\}$ \emph{\sinv{}} if $a(2n) = a(n)$ for all $n \in \NN_0$, or equivalently, if $a = a_A$ for a set $A$ that contains only patterns that begin and end with $\mathtt{1}$ (see Section \ref{ssec:Invariant} for details). In the \sinv{} case, we have a conjectural classification, which we are able to confirm in one direction in full generality, and in the opposite direction for patterns of length $\leq 5$. 

\begin{alphatheorem}\label{thm:C}
	Let $A$ be a set of patterns over the alphabet $\{\mathtt 0, \mathtt 1\}$, all of which begin and end with $\mathtt 1$. Let $\ell$ be the length of the longest word in $A$ and let $a = a_A$ be the corresponding binary \pc{} sequence. If $\ell \geq 2$ and $\mathtt{1} \{\mathtt 0,\mathtt 1\}^{\ell -2} \mathtt{1} \subset A$ then $a$ is noncorrelated. Conversely, if $2 \leq \ell \leq 5$ and $a$ is noncorrelated then $\mathtt{1} \{\mathtt 0, \mathtt1\}^{\ell -2} \mathtt{1} \subset A$.
\end{alphatheorem}

\begin{conjecture}\label{conj:NC+inv->sat}
	Let $A$, $\ell$ and $a$ be as in Theorem \ref{thm:C}. If $a$ is noncorrelated then $\ell \geq 2$ and $\mathtt{1} \{\mathtt 0, \mathtt1\}^{\ell -2} \mathtt{1} \subset A$.
\end{conjecture}

If $A = \mathtt{1}\{\mathtt{0},\mathtt{1}\}^{\ell-2}\mathtt{1}$ then the fact that $a_A$ is noncorrelated follows from \cite{ZhengPengKamae-2018}. More generally, Theorem 1.3 in \cite{ZhengPengKamae-2018} (see also \cite{AlloucheLiardet-1991}) provides a classification of all noncorrelated binary \pc{} sequences $a_A$ for sets of patterns of the form $A = w_1 \{\mathtt 0,\mathtt 1\}^{\ell_1} w_2  \{\mathtt 0,\mathtt 1\}^{\ell_2} \dots w_{s}  \{\mathtt 0,\mathtt 1\}^{\ell_s} w_{s+1}$ where $s \in \NN$, $w_i$ are words over the alphabet $\{\mathtt 0,\mathtt 1\}$ and $\ell_1,\ell_2, \dots, \ell_s \in \NN_0$. Conjecture \ref{conj:NC+inv->sat} is consistent with said classification.

Returning to the general case, we notice that each binary \pc{} sequence $a \colon \NN_0 \to \{+1,-1\}$ can be written as the product of a periodic sequence $h$ and a \sinv{} \pc{} sequence $b$ (Lemma \ref{lem:PC-rep-invariant}). The correlation coefficients of $a$ and $b$ are closely related (see also Remark \ref{rmk:final}), and in all cases that we were able to check (i.e., $\ell \leq 4$), if $a$ is noncorrelated then so is $b$. This motivates us to put forward the following conjecture.

\begin{conjecture}\label{conj:NC->inv_x_per}
	Let $a \colon \NN_0 \to \{+1,-1\}$ be a noncorrelated binary \pc{} sequence. Then $a$ is the product of a periodic sequence and an \sinv{} noncorrelated binary \pc{} sequence
\end{conjecture} 

Above we restricted our attention to base $2$ for the sake of brevity. In the remaining part of the paper, we work in arbitrary base $k \geq 2$. In particular, the natural base-$k$ variant of Theorem \ref{thm:B} holds true. The same applies to the first part of Theorem \ref{thm:C}, except that it is less clear what the base-$k$ variant should be and the resulting statement is vacuous for many values of $k$ (see Proposition \ref{prop:sat->NC}). When it comes to computations, we only consider base $2$ since for larger bases the number of distinct \pc{} sequences  becomes so large that merely listing them all is already infeasible even for modest pattern lengths.

\subsection*{Acknowledgements}
While writing this paper, the author was supported by the ERC grant ErgComNum 682150 at the Hebrew University of Jerusalem. During the review process, the author was working within the framework of the LABEX MILYON (ANR-10-LABX-0070) of Universit\'{e} de Lyon, within the program "Investissements d'Avenir" (ANR-11-IDEX-0007) operated by the French National Research Agency (ANR). The author also acknowledges support from the Foundation for Polish Science (FNP).

The author wishes to express his gratitude to Boris Adamczewski, Jakub Byszewski, Aihua Fan and Tamar Ziegler for helpful conversations and to the anonymous Referee for the careful reading of this paper and constrictive suggestions.
 
\section{Background and definitions}\label{sec:Background}

\color{good}
\textit{\textbf{Convention:} Throughout the paper, $k$ denotes the base and is considered to be fixed. In particular, all constructions and constants are allowed to depend on $k$ unless explicitly stated otherwise.}

\subsection{\Pc{} sequences}\color{final}\label{ssec:PatternCountingDefs} 
We let $\Sigma_k = \{\mathtt 0, \mathtt 1,\dots,k-1\}$ denote the set of digits in base $k$. For a set $X$, we let $X^*$ denote the monoid consisting of words over the alphabet $X$, equipped with the operation of concatenation and neutral element $\epsilon$, the empty word. For $v \in X^*$, $\abs{v}$ denotes the length of $v$. For $n \in \NN_0$, $(n)_k \in \Sigma_k^*$ denotes the expansion of $n$ in base $k$ (without leading zeros). Conversely, for $v \in \Sigma_k^*$, $[v]_k \in \NN_0$ denotes the integer encoded by $v$.

Let $X$ be a set. We say that a word $v \in X^*$ appears in another word $w \in X^*$, or that $v$ is a factor of $w$, if there exist $x,y \in X^*$ such that $w = xvy$. We call $v$ a prefix (resp.{} suffix) of $w$ if we may take $x = \epsilon$ (resp.{} $y = \epsilon$). We further define $\#(v,w)$ to be the number of times $v$ appears in $w$, that is, the number distinct of pairs $(x,y) \in X^* \times X^*$ such that $w = xvy$. We note that this definition allows for overlaps, so for instance $\#(\mathtt{010},\mathtt{01010}) = 2$. More generally, for a finite set $A \subset X^*$, we define $\#(A,w) = \sum_{v \in A} \#(v,w)$.

Accordingly, for $n \in \NN_0$ and $v \in \Sigma_k^* \setminus \{\mathtt 0\}^*$, $\#(n,v)$ denotes the number of times that $v$ appears in the base-$k$ expansion of $n$ padded with sufficiently many leading zeros, that is, $\#(v,n) = (v,\mathtt 0^{\abs{v}-1} (n)_k)$. The inclusion of the leading zeros in the expansion of $n$ ensures better behaviour of the map $n \mapsto \#(v,n)$; in particular, for each $n,m \in \NN_0$ and sufficiently large $\alpha \in \NN_0$ we have $\#(v,k^\alpha m + n) = \#(v,k^\alpha m)+\#(v, n)$.
The assumption that $v$ is not a string of zeros ensures that $\#(v,n)$ is well-defined, in the sense that for fixed $n$, the sequence $(v,\mathtt 0^{\alpha} (n)_k)$ ($\alpha \in \NN_0$) is eventually constant. 

We will call a set $A \subset \Sigma_k^*$ \emph{admissible} if $A$ is finite and $A \cap \{\mathtt 0\}^* = \emptyset$, so that we may define $\#(A,n) = \sum_{v \in A} \#(v,n)$. For any admissible set $A$, the corresponding \emph{\pc{} sequence} is defined by (cf.{} \cite[Definition 1.1]{ZhengPengKamae-2018})
\begin{equation}\label{eq:def-of-a_A}
	a_{A}(n) = (-1)^{\#(A,n)}.
\end{equation}
If additionally $\abs{u} \leq \ell$ for all $u \in A$ then we say that $a$ is a \pc{} sequence of length $\leq \ell$, or equivalently we define the \emph{length} of $a$ as the least possible value of $\max_{u \in A} \abs{u}$ among all representations of $a$ in the form \eqref{eq:def-of-a_A}, where $A \subset \Sigma_k^*$ is an admissible set. Note that one \pc{} sequence can have multiple representations of the aforementioned form.

For two sets $A,B$, we let $A \oplus B$ denote the symmetric difference $(A\setminus B) \cup (B\setminus A)$.

\begin{lemma}\label{lem:PC-is-closed}
	The class of \pc{} sequences $\NN_0 \to \{+ 1, -1\}$ is closed under multiplication.
\end{lemma}
\begin{proof}
	It is enough to note that for any admissible sets $A,B \subset \Sigma_k^*$ we have 
	\[ a_A \cdot a_B = a_{A \oplus B}.\qedhere\]
\end{proof}

It will usually be convenient to impose further restrictions on the set of patterns $A$. Depending on the context we require that either $A$ has not leading zeros (in the sense that that $\mathtt{0}$ is not a prefix of $v$ for any $v \in A$) or that $A$ has constant length (in the sense that there is some $\ell \in \NN$ such that $\abs{v} = \ell$ for all $v \in A$).
\begin{lemma}\label{lem:PC-representation}
	Let $\ell \in \NN_0$ and let $a \colon \NN_0 \to \{+1,-1\}$ be a \pc{} sequence of length $\leq \ell$. Then there exist admissible sets $B,C \subset \Sigma_k^*$ such that $B$ has no leading zeros, $C$ has constant length $\ell$, and $a = a_B = a_C$. Moreover, $B$ and $C$ are uniquely determined by $a$. 
\end{lemma}
\begin{proof}
	Pick any admissible set $A \subset \Sigma_k^*$ with $a = a_A$.
	Note that for each $v \in \Sigma_k^*$ and each $n \in \NN_0$ we have 
	\begin{equation}\label{eq:161:1}
		\#(v,n) = \sum_{i = 0}^{k-1} \#(iv,n)
	\end{equation}
	
	To construct $B$, begin with $A$ and as long as $A$ contains at least one word starting with $\mathtt 0$, say $\mathtt 0 v$, replace $A$ with $A = A \oplus \set{ i v}{i \in \Sigma_k }\oplus \{v\}$. Because of \eqref{eq:161:1}, this operation does not change the sequence $a_A$. Since each iteration decreases the total number of leading zeros in the patterns in $A$, after a finite number of steps this procedure must terminate and the resulting set of patterns has no leading zeros.
		
	 To construct $C$, likewise, begin with $A$ and as long as $A$ contains at least one word $v$ with length $\abs{v} < \ell$, pick the shortest such word $v$ and replace $A$ with $A \oplus \set{ i v}{i \in \Sigma_k }\oplus \{v\}$. Like before, this operation does not change the sequence $a_A$. Each iteration either decreases the number of words in $A$ with least possible length, or increases the length of the shortest word in $A$. At the same time, no words of length larger than $\ell$ are introduced. Hence, after a finite number of steps this procedure must terminate and the resulting set of patterns has constant length equal to $\ell$. 
	 
	 It remains to show uniqueness. Using Lemma \ref{lem:PC-is-closed}, we may assume that $A = \emptyset$. For the sake of contradiction, suppose that one of $B$ and $C$ is non-empty. Consider first the case when $B \neq \emptyset$ and let $v$ be the shortest word in $B$. Then $1 =a_B([v]_k) = -1$, since $\mathtt 0^\ell v$ contains exactly one pattern from $B$, namely $v$. Hence, we have reached a contradiction. Consider next the case when $C \neq \emptyset$ and choose the word $\mathtt 0^m v \in C$ where $m$ is largest possible. Then we again reach the contradiction: $1 = a_C([v]_k) = -1$.  
\end{proof}

\begin{remark}\color{newgood}
	We focus our attention on $\pm 1$-valued sequences for two basic reasons. The first one is practical: The noncorrelation phenomenon that we are interested in relies on occurrence of certain arithmetic coincidences, which become less likely as the number of possible values increases; accordingly, the computational part of the problem becomes increasingly resource-intensive as sequences under consideration become more complicated. The second reason is conceptual: For a $\pm 1$-valued sequence $a$ with mean $\M{a} = 0$, noncorrelation is tantamount to equidistribution of the pairs $a(n),a(n+m)$. More precisely, for each $m \in \NN$, if we additionally assume that the limits mentioned above exist then $\gamma_{a}(m) = 0$ if and only if
	\[
		  \lim_{N \to \infty}\frac{\abs{\set{n < N}{a(n) = i,\ a(n+m) = i'}}}{N} =\frac{1}{4} \quad \text{ for all } i,i' \in \{+1,-1\}.
	\]
	The analogous characterisation is false without the assumption that $a$ is allowed to take more than $2$ values.
\end{remark}

\subsection{Automatic sequences}
In this section we briefly discuss the basics of the theory of automatic sequences; for extensive background see \cite{AlloucheShallit-book}.
For $i \in \Sigma_k$, we define the operators $\Lambda_i$ acting on sequences $\NN_0 \to \CC$ by
\begin{equation}
\label{eq:def-of-Lambda}
\Lambda_i a (n) = a(kn+i). 
\end{equation}
The \emph{$k$-kernel} of a sequence $a \colon \NN_0 \to \CC$ consists of all sequences $\NN_0 \to \CC$ that can be obtained from $a$ by repeated application of $\Lambda_i$'s, that is, 
\begin{equation}
\label{eq:def-of-kernel}
\cN_k(a) = \set{ n \mapsto a(k^\alpha n + r)}{ \alpha,r \in \NN_0, 0 \leq r < k^\alpha} \subset \CC^{\NN_0}. 
\end{equation}
It will also be convenient to introduce the shift operator $S$ acting on sequences $\NN_0 \to \CC$ by $Sa(n) = a(n+1)$. For future reference, we record how the introduced operators interact.
\begin{lemma}\label{lem:ops-compose}
	For each $0 \leq i < k-1$ we have $\Lambda_i S = \Lambda_{i+1}$. Moreover, $\Lambda_{k-1} S = S \Lambda_1$. 
\end{lemma}
\begin{proof}
	Direct computation.
\end{proof}

A sequence $a$ is $k$-automatic (or just automatic, if $k$ is clear from the context) if and only if $\cN_k(f)$ is finite. Many equivalent definitions of automaticity are possible, and we briefly mention some of them to provide context. Details and terminology can be found in \cite{AlloucheShallit-book}. As the name suggests, a sequence is $k$-automatic if and only if it is computed by a deterministic finite $k$-automaton with output. Any fixed point of a $k$-uniform morphism is $k$-automatic, and conversely any $k$-automatic sequence can be obtained as a letter-to-letter coding of a fixed point of a $k$-uniform morphism. When $k$ is a prime and $a$ is a sequence taking values in a finite field $\mathbb{F}$ of characteristic $k$, yet another criterion due to Christol shows that $a$ is automatic if and only if the associated formal power series is algebraic over $\mathbb{F}$. 

It is a well-known fact that the class of $k$-automatic complex-valued sequences is closed under addition, multiplication, conjugation and restriction to subsequences, that is, if $a,b \colon \NN_0 \to \CC$ are $k$-automatic, then so are $n \mapsto a(n) + b(n)$, $n \mapsto a(n) \cdot b(n)$, $n \mapsto \bar a(n)$ and $n \mapsto a(An+B)$ for any $A \in \NN$, $B \in \NN_0$. More generally, if $a,b \colon \NN_0 \to \CC$ are $k$-automatic and $h \colon \CC^2 \to \CC$ is arbitrary, then the sequence $n \mapsto h(a(n),b(n))$ is $k$-automatic.

{\color{final}
For a sequence $a \colon \NN_0 \to \CC$, we define the mean and the logarithmic mean:
\begin{equation}\label{eq:def-of-M}
\M{a} := \lim_{N \to \infty} \frac{1}{N} \sum_{n=0}^{N-1} a(n), \ \qquad
\Mlog{a} := \lim_{N \to \infty} \frac{1}{\log N} \sum_{n=0}^{N-1} \frac{1}{n+1}a(n).
\end{equation}
We note that $\M{a}$ are not guaranteed to exist, even when the sequence $a$ is automatic. (Consider, for instance, the sequence defined by $a(0) = 0$ and $a(n) = (-1)^\alpha$ if $k^\alpha \leq n < k^{\alpha+1}$, $\alpha \in \NN_0$.) On the other hand, we have the following positive result for logarithmic means.
\begin{theorem}[{\cite[Thm.{} 8.4.8]{AlloucheShallit-book}}]\label{thm:log-mean-exists}
	Let $a \colon \NN_0 \to \CC$ be a $k$-automatic sequence. Then $\Mlog{a}$ exists.
\end{theorem} 
We also record the fact that that if $a \colon \NN_0 \to \CC$ is a bounded sequence and $\M{a}$ exists then $\Mlog{a}$ also exists and $\Mlog{a} = \M{a}$, see e.g.{} \cite[Prop.{} 8.4.4 (a)]{AlloucheShallit-book}.
}

\Pc{} sequences are, unsurprisingly, automatic. In fact, we have the following characterisation of \pc{} sequences in terms of their $k$-kernels (cf.{} \cite[Lemma 2.2]{ZhengPengKamae-2018}).
\begin{lemma}\label{lem:char-of-PC}\color{final}
	Let $a \colon \NN_0 \to \{+1,-1\}$ be a sequence with $a(0) = +1$ and $\ell \in \NN$. Then the following conditions are equivalent:
	\begin{enumerate}
	\item\label{it:631:A} There exists a set $A \subset \Sigma_k^\ell \setminus \{\mathtt{0}^\ell \}$ with $a = a_A$.
	\item \label{it:631:B} For each $b \in \cN_k(a)$, the sequence $b/a$ has period $k^{\ell-1}$.\end{enumerate}
\end{lemma}
\begin{proof}\color{final}
	\eqref{it:631:A} $\Rightarrow$ \eqref{it:631:B}: Let $i \in \Sigma_k$ and $n \in \NN_0$. Then each factor of $(n)_k$ is also a factor of $(kn+i)_k = (n)_ki$ and conversely each factor of $(kn+i)_k$ that is not a suffix is a factor of $(n)_k$. More precisely, for each $v \in \Sigma_k^* \setminus \{\mathtt 0\}^*$ we have
\[	
	\#(v,kn+i) = 
	\begin{cases} 
	\#(v,n) + 1 &\text{if } v \text{ is a suffix of }\mathtt{0}^{\abs{v}-1}(kn+i)_k,\\
	\#(v,n) &\text{ otherwise.}
	\end{cases}
\]
Consequently, $\Lambda_i a(n)/a(n) = -1$ if the suffix of $\mathtt 0^{\ell-1}(kn+i)_k$ of length $\ell$ belongs to $A$ and $\Lambda_i a(n)/a(n) = +1$ otherwise. It follows that $h_i := \Lambda_i a/a$ is $k^{\ell-1}$-periodic. Since for each $i \in \Sigma_i$, the operator $\Lambda_i$ maps $k^{\ell-1}$-periodic sequences to $k^{\ell-2}$-periodic sequences (or constant sequences, if $\ell = 1$), it follows that all sequences in the $k$-kernel of $a$ take the form $a \cdot h$ where $h$ is $k^{\ell-1}$-periodic.

	\eqref{it:631:B} $\Rightarrow$ \eqref{it:631:A}: For each $i \in \Sigma_k$, let $h_i := \Lambda_i a/a$. Note that $h_0(0) = a(0)/a(0) = +1$ and that the sequences $h_i$ take values in $\{+1,-1\}$. Conversely, given any $k$-tuple of $k^{\ell-1}$-periodic $\{+1,-1\}$-valued sequences $h_i'$ ($i \in \Sigma_k$) with $h_0'(0) = +1$, we can inductively construct a sequence $a \colon \NN_0 \to \{+1,-1\}$ with $a(0) = +1$ and $h_i = h_i'$ for all $i \in \Sigma_k$. Hence, the number of sequences that satisfy \eqref{it:631:B} is 
	\[
		\bra{ 2^{k^{\ell-1}} }^{k-1} \cdot 2^{k^{\ell-1}-1} = 2^{k^\ell -1}.
	\]
On the other hand, the number of subsets of $ \Sigma_k^\ell \setminus \{\mathtt{0}^\ell \}$ is also equal to $2^{k^\ell-1}$, and by the previously proven implication and Lemma \ref{lem:PC-representation}, each of these choices gives rise to a different sequence satisfying \eqref{it:631:B}. It follows that each sequence satisfying \eqref{it:631:B} has a representation as in \eqref{it:631:A}.
\end{proof}

{\color{final}
\subsection{Regular sequences}
The class of $k$-regular sequences was introduced by Allouche and Shallit \cite{AlloucheShallit-1992,AlloucheShallit-2003} as a natural generalization of the class of $k$-automatic sequences.

Let $R$ be a ring contained in $\CC$. A sequence $f \colon \NN_0 \to \CC$ is \emph{$(R,k)$-regular} if $\cN_k(f)$ is contained in a finitely generated $R$-module. Note that if $R' \subset \CC$ is another ring and $R \subset R'$ then any $(R,k)$-regular sequence is also $(R',k)$-regular. In our context, the choice of the ring $R$ does not play a major role: For the sake of brevity, we set $R = \QQ$ throughout the paper and omit $R$ from the notation. (Strictly speaking we could have worked with $R = \ZZ[1/k]$, making some results marginally stronger.) The fact that the ring under consideration is in fact a field leads to a slightly more succinct definition of regularity: A sequence $f \colon \NN_0 \to \CC$ is $k$-regular if and only if its $k$-kernel spans a finite dimensional vector space over $\QQ$: $\dim \linspan_{\QQ} \cN_k(f) < \infty$.

The class of $k$-regular sequences enjoys closure properties analogous to $k$-automatic sequences: If $f,g \colon \NN_0 \to \CC$ are $k$-regular, then so are $n \mapsto f(n) + g(n)$, $n \mapsto f(n) + g(n)$, $n \mapsto \bar f(n)$, $n \mapsto z f(n)$ ($z \in \CC$) and $n \mapsto f(An+B)$ ($A \in \NN,\ B \in \NN_0$). In particular, $k$-regular sequences $\NN_0 \to \CC$ form an involutive algebra over $\CC$ (with addition and multiplication defined pointwise). 

We will need a method to verify if a given regular sequence is identically zero. The following lemma provides a simple criterion.

\begin{lemma}\label{lem:zero-criterion}\color{final}
	Let $f \colon \NN_0 \to \CC$ be $k$-regular and non-zero. Then there exists $g \in \cN_k(f)$ with $g(0) \neq 0$.
\end{lemma}
\begin{proof}\color{final}
	For the sake of contradiction, suppose that $g(0) = 0$ for all $g \in \cN_k(f)$. We show by induction on $\alpha$ that $g(n) = 0$ for all $g \in \cN_k(f)$ and $0 \leq n < k^\alpha$. If $\alpha = 0$ then $n = 0$, so there is nothing to prove. If $\alpha > 0$ and $n < k^\alpha$ then $n = kn' + i$ where $i \in \Sigma_k$ and $n' < k^{\alpha-1}$. Hence, $g(n) = \Lambda_i g(n') = 0$ by the inductive assumption.
\end{proof}
}

\subsection{Invariant sequences}\label{ssec:Invariant}
We will say that a sequence $a \colon \NN_0 \to \CC$ is \emph{\sinv{}} if $a(kn) = a(n)$ for all $n \in \NN_0$. The \sinv{} \pc{} sequences admit a simple description. Following the convention in Section \ref{ssec:PatternCountingDefs}, we will say that a set $A \subset \Sigma_k^*$ has no trailing zeros if $\mathtt{0}$ is not a suffix of any $v \in A$. 

\begin{lemma}\label{lem:PC-crit-invariant}\color{final}
	Let $a \colon \NN_0 \to \{+1,-1\}$ be a \pc{} sequence. Then $a$ is \sinv{} if and only if there exists a set $A \subset \Sigma_k^*$ that has no leading and no trailing zeros and such that $a = a_A$.
\end{lemma}
\begin{proof}\color{good}
	If $A \subset \Sigma_k^*$ has no trailing zeros then $\#(v,n) = \#(v,kn)$ for all $v \in A$ and $n \in \NN_0$, so $a_A$ is \sinv{}.
	
	Conversely, suppose that $a$ is \sinv{} and let $A \subset \Sigma_k^*$ be a set of patterns without leading zeros such that $a = a_A$, which exists by Lemma \ref{lem:PC-representation}. Suppose for the sake of contradiction that $A$ contains a pattern ending with $\mathtt{0}$, say $u \mathtt{0} \in A$ for some $u \in \Sigma_k^*$, and let $u$ be as short as possible. Since $a$ is \sinv{}, we have
\begin{equation}\label{eq:769:1}
		a([u\mathtt{0}]_k) = a([u]_k).
\end{equation}
On the other hand, each $v \in A$ either ends in a non-zero digit (in which case $\#(v,u\mathtt{0}) = \#(v,u)$), or ends in $\mathtt{0}$ and is not a factor or $u\mathtt{0}$  (in which case $\#(v,u\mathtt{0}) = \#(v,u) = 0$), or is equal to $u\mathtt{0}$ (in which case $\#(v,u\mathtt{0}) = 1$ and $\#(v,u) = 0$). As a consequence,
\begin{equation}\label{eq:769:2}
	a([u\mathtt{0}]_k) = (-1)^{\#(A,u\mathtt{0})} = (-1)^{\#(A,u)+1}= -a([u]_k),
\end{equation}
which contradicts \eqref{eq:769:1} and finishes the argument.
\end{proof}

We also record the fact that every \pc{} sequence is the product of a \sinv{} sequence and a periodic sequence. As we will see (cf.{} Remark \ref{rmk:final}) the introduction of the multiplicative factor affects the correlation coefficients in a relatively simple way, which motivates our focus on \sinv{} sequences.

\begin{lemma}\label{lem:PC-rep-invariant}\color{good}
	Let $\ell \in \NN_0$ and let $a \colon \NN_0 \to \{+1,-1\}$ be a \pc{} sequence of length $\leq \ell$. Then there exist a unique \sinv{} \pc{} sequence of length $\leq \ell$ such that $a/b$ is $k^{\ell-1}$-periodic.
\end{lemma}
\begin{proof}
	By Lemma \ref{lem:PC-representation}, we may assume that $a = a_A$ for a set $A \subset \Sigma_k^*$ without leading zeros. Reasoning along similar lines as in the proof of Lemma \ref{lem:PC-representation}, we note that for any word $v$, we have
\[
	\sum_{i=0}^{k-1}\#(vi,n) - \#(v,n) = 
	\begin{cases}
	1 &\text{ if } v \text{ is a suffix of } n,\\
	0 &\text{ otherwise.}
	\end{cases} 
\]
In particular, letting $D(v) := \set{vi}{i \in \Sigma_k}\cup \{v\}$, we see that the sequence $a_{D(v)}$ is $k^{\abs{v}}$-periodic. We construct a sequence of sets $A := A_0,A_1,\dots, A_t = :B$, where $A_{j+1} = A_j \oplus D(v)$ if $A_j$ contains the word $v\mathtt 0$ for some $v \in \Sigma_k^*$ and $t$ is the first index such that no word in $A_t$ ends with $\mathtt{0}$. This construction is guaranteed to terminate because each step decreases the total length of words in $A_j$ that end with $\mathtt{0}$. Letting $b :=a_B$ we observe that $a/b$ is the product of $k^{\ell-1}$-periodic sequences and hence $k^{\ell-1}$-periodic.
\end{proof}

\section{Correlation coefficients}\label{sec:Correlation}

In this section we study correlation coefficients of $k$-automatic sequences and show that they are $k$-regular (Corollary \ref{cor:gamma-is-reg}). This allows us to reduce the task of verifying if a given $k$-automatic sequence is noncorrelated to checking if a given $k$-regular sequence is identically zero on $\NN$, which can be accomplished with the help of Lemma \ref{lem:zero-criterion}.

{\subsection{Definitions}\color{final}
For two sequences $a,b \colon \NN_0 \to \CC$, we define the correlation coefficients:
\begin{equation}
\label{eq:def-of-gamma}
	\gamma_{a,b}(m) := \lim_{N \to \infty} \frac{1}{N} \sum_{n=0}^{N-1} a(n) \bar b(n+m),
\end{equation}
if the limit exists (otherwise, $\gamma_{a,b}(m)$ is considered undefined). We are often interested in the case where $a = b$, when we write $\gamma_a$ in place of $\gamma_{a,a}$.
Unfortunately, the limit defining $\gamma_{a,b}(m)$ is not guaranteed to converge even if $a$ and $b$ are automatic. This motivates us to consider the logarithmic correlation coefficients, defined by
\begin{equation}
\label{eq:def-of-gammalog}
	\gammalog_{a,b}(m) := \lim_{N \to \infty} \frac{1}{\log N} \sum_{n=0}^{N-1} \frac{1}{n+1} a(n) \bar b(n+m),
\end{equation}

If $a$ and $b$ are automatic and $m \in \NN_0$, then the sequence $n \mapsto a(n)\bar b(n+m)$ is also automatic. Since Theorem \ref{thm:log-mean-exists} guarantees existence of logarithmic means of automatic sequences, we have the following fact.
\begin{corollary}
	Let $a,b \colon \NN_0 \to \CC$ be $k$-automatic sequences. Then the coefficients $\gammalog_{a,b}(m)$ are well-defined for all $m \in \NN_0$. Moreover, if the coefficient $\gamma_{a,b}(m)$ is well-defined for some $m \in \NN_0$ then $\gamma_{a,b}(m) = \gammalog_{a,b}(m)$.
\end{corollary}
}

\subsection{Recurrence}
{\color{final}
Our next goal is to obtain a recursive description of the correlation coefficients discussed above. Recall that for a $k$-automatic sequence $a$, the kernel $\cN_k(a)$ is finite and closed under the operators $\Lambda_i$ defined in \eqref{eq:def-of-Lambda} for all $i \in \Sigma_k$.
}

\begin{lemma}\label{lem:gammalog-recurrence}\color{final}
	Let $\cN$ be a finite set of sequences $\NN_0 \to \CC$, closed under the operators $\Lambda_i$ for all $i \in \Sigma_k$. Then for all $a,b \in \cN$ it holds that
	\begin{equation}\label{eq:791:1}
		\gammalog_{a,b}(m) = \frac{1}{k} \sum_{i=0}^{k-1} \gammalog_{a_i',b_i'}(m_i'), 
	\end{equation}
	where $a_i',b_i'$ and $m_i'$ are given by
	\begin{equation}\label{eq:791:2}
		a_i' = \Lambda_{i} a,\quad b_i' = \Lambda_{ m+i \bmod{k} } b, \quad m_i' = \floor{\frac{m+i}{k}}.
	\end{equation}
\end{lemma}\color{final}
\begin{proof}  Rescaling if necessary, we assume that all sequences in $\cN$ are $1$-bounded (that is, $\abs{a(n)} \leq 1$ for all $a \in \cN$ and $n \in \NN_0$).
	For each $N > 0$, splitting $[N]$ into residue classes modulo $k$ we obtain
	\begin{align*}
		\sum_{n=0}^{N-1} \frac{a(n) \bar b(n+m) }{n+1} 
		&= \sum_{i=0}^{k-1} \sum_{n=0}^{\floor{N/k}-1} \frac{a(kn+i)\bar b(kn+i+m)}{kn+i+1}  + O\bra{\frac{1}{N}}.
		\\&=  \frac{1}{k} \sum_{i=0}^{k-1} \sum_{n=0}^{\floor{N/k}-1} \frac{a_i'(n) \bar b_i'(n+m_i')}{n+1} + O\bra{1},
	\end{align*}
	where $a_i',b_i'$ and $m_i'$ are given by \eqref{eq:791:2}, and we use the estimate ${1}/\bra{kn+i+1} = {1}/{k(n+1)} + O(1/(n+1)^2)$ together with the fact that $\sum_{n=0}^\infty 1/(n+1)^2$ is summable. Dividing by $\log N$ and recalling that $1/\log(N/k) = 1/\log N + O(1/\log^2 N)$ we obtain
	\begin{align*}
		\frac{1}{\log N} \sum_{n=0}^{N-1} \frac{a(n) \bar b(n+m) }{n+1} 
		&=  \frac{1}{k} \sum_{i=0}^{k-1} \frac{1}{\log\floor{N/k}} \sum_{n=0}^{\floor{N/k}-1} \frac{a_i'(n) \bar b_i'(n+m_i')}{n+1} + O\bra{\frac{1}{\log N}}.
	\end{align*}
	Letting $N \to \infty$, we obtain \eqref{eq:791:1}.
\end{proof}

\color{new}
{\color{final}
While the coefficients $\gammalog_{a,b}(m)$ are better-behaved in general, our original motivation concerns the coefficients $\gamma_{a,b}(m)$ (where additionally $a = b$). Fortunately, existence of the latter is easy to ensure under mild additional assumptions.
}

\begin{lemma}\label{lem:gamma-exists}\color{final}
	Let $\cN$ be a finite set of sequences $\NN_0 \to \CC$, closed under the operators $\Lambda_i$ for all $i \in \Sigma_k$. Suppose that $\gamma_{a,b}(0)$ exists for each $a,b \in \cN$. Then also $\gamma_{a,b}(m)$ exists for all $a,b \in \cN$ and $m \in \NN_0$ and, using the notation from \eqref{eq:791:2}, satisfy
	\begin{equation}\label{eq:791:1-b}
		\gamma_{a,b}(m) = \frac{1}{k} \sum_{i=0}^{k-1} \gamma_{a_i',b_i'}(m_i'), 
	\end{equation}	
\end{lemma}\color{final}
\begin{proof} Rescaling if necessary, we may assume that all sequences in $\cN$ are $1$-bounded.
	Generalizing the definition of $\gamma_{a,b}(m)$ slightly, for $x \geq 1$ let us put
	\begin{equation}\label{eq:427:01}
		\gamma_{a,b}(m;x) = \frac{1}{\floor{x}} \sum_{n=0}^{\floor{x}-1} a(n) \bar b(n+m).	
	\end{equation}
	Then, following the same reasoning as in Lemma \ref{lem:gammalog-recurrence}, we find the recursive relation
	\begin{equation}\label{eq:427:02}
		\gamma_{a,b}(m;x) = \frac{1}{k} \sum_{i=0}^{k-1} \gamma_{a_i',b_i'}(m_i';x/k) + O(1/x), 
	\end{equation}
	where $a_i',b_i'$ and $m_i'$ are given by \eqref{eq:791:2}. 
	
	In particular, for $m = 1$ we obtain
	\begin{equation}\label{eq:427:03}
		\gamma_{a,b}(1;x) = \frac{1}{k} \sum_{i=0}^{k-2} \gamma_{a_i',b_i'}(0;x/k) + \frac{1}{k} \gamma_{a_{k-1}',b_{k-1}'}(1;x/k) + O(1/x). 
	\end{equation}
	Iterating \eqref{eq:427:03} $t$ times, we conclude that there exist weights $w_{a',b'}^{(t)} \geq 0$ ($a',b' \in \cN$) with $\sum_{a',b' \in \cN} w_{a',b'}^{(t)} = 1-1/k^t$ and sequences $a^{(t)},b^{(t)} \in \cN$ such that
	\begin{equation}\label{eq:427:04}
		\gamma_{a,b}(1;x) = \sum_{a',b' \in \cN} w_{a',b'}^{(t)} \gamma_{a',b'}(0;x/k^t) + \frac{1}{k^t} \gamma_{a^{(t)},b^{(t)}}(1;x/k^t) + O(k^t/x). 
	\end{equation}
	Since $\gamma_{a',b'}(0;y) \to \gamma_{a',b'}(0)$ as $y \to \infty$ for each $a',b' \in \cN$, letting $x \to \infty$ in \eqref{eq:427:04} we conclude that there exists a number $\gamma_{a,b}^{(t)}(1) = \sum_{a',b' \in \cN} w_{a',b'}^{(t)} \gamma_{a',b'}(0)$ such that
	\begin{equation}\label{eq:427:05}
		\limsup_{x \to \infty} \abs{ \gamma_{a,b}(1;x) -\gamma_{a,b}^{(t)}(1)} = O(1/k^t). 
	\end{equation}	 
	It follows that the sequence $\gamma_{a,b}^{(t)}(1)$ ($t \in \NN$) is Cauchy, and $\gamma_{a,b}(1)$ is well-defined:
	\begin{equation}\label{eq:427:06}
	\gamma_{a,b}(1) = \lim_{x \to \infty} \gamma_{a,b}(1;x)= \lim_{t \to \infty} \gamma_{a,b}^{(t)}(1).
	\end{equation}	 

	We are now ready to prove by induction on $m$ that the coefficients $\gamma_{a,b}(m)$ are well-defined for all $m \in \NN_0$ and $a,b \in \cN$. The case $m = 0$ is included in the assumptions, and we have dealt with $m = 1$ above. Suppose now that $m \geq 2$. For each $i \in \Sigma_k$, since $\floor{x/k} \leq x/k < x$ for all $x > 0$, we have
	\begin{equation}\label{eq:427:07}
		m_i' = \floor{ \frac{m+i}{k}} \leq \floor{ \frac{m+k-1}{k}} = \floor{ \frac{m-1}{k}} + 1 < (m-1) + 1 = m.
	\end{equation}	
Hence, existence of $\gamma_{a,b}(m)$ follows from \eqref{eq:427:03} and the inductive assumption. Finally, to obtain \eqref{eq:791:1-b} it remains to pass to the limit $x \to \infty$ in \eqref{eq:427:02} (or use Lemma \ref{lem:gammalog-recurrence} combined with the remark after Theorem \ref{thm:log-mean-exists}).
\end{proof}

{\subsection{Regularity}\color{final}
  We are now ready to show that the logarithmic correlation sequences coming from $k$-automatic sequences are $k$-regular. In fact, bearing in mind applications in Section \ref{sec:Implementation} we record a slightly more precise statement. Recall that for a sequence $a$, the sequence $Sa$ is given by $Sa(n) = a(n+1)$. Similar ideas can be seen in \cite[Thm.{} 6]{AlloucheShallit-2003}.

\begin{proposition}\label{prop:gamma-is-reg}\color{final}
	Let $\cN$ be a finite set of sequences $\NN_0 \to \CC$, closed under the operators $\Lambda_i$ for all $i \in \Sigma_k$. Let $\cM = \set{ S^{e} \gammalog_{a,b} }{a,b \in \cN, e \in \{0,1\} } $. Then $\linspan_{\QQ} \cM$ is closed under the operators $\Lambda_i$ for all $i \in \Sigma_k$. 
\end{proposition}
\begin{proof}\color{final}
	Pick any $g = S^{e} \gammalog_{a,b} \in \cM$ ($a,b \in \cN$, $e \in \{0,1\}$) and $j \in \Sigma_k$. It follows from Lemma \ref{lem:gammalog-recurrence} that
\begin{equation}\label{eq:202:1}
		\Lambda_j g(n) = \gammalog_{a,b}(kn+j+e) = \frac{1}{k} \sum_{i=0}^{k-1} \gamma_{a_i',b_i'}^{\log}(n + e_i')
\end{equation}
where for each $i \in \Sigma_k$, $a_i',b_i' \in \cN$ and 
\[
e_i' = \floor{ \frac{j+i+e}{k} } \leq \floor{ \frac{(k-1)+(k-1)+1}{k} } = 1.
\]
It remains to note that each of the functions of $n$ appearing under the sum on the right hand side of \eqref{eq:202:1} belongs to $\cM$.
\end{proof}

\begin{theorem}\label{cor:gamma-is-reg}	\color{final}
If $a \colon \NN_0 \to \CC$ is $k$-automatic then the sequence $\gamma_{a}^{\log}$ is $k$-regular and $\dim \linspan_{\QQ} \cN_k(\gamma_{a,a}^{\log}) \leq 2 \abs{\cN_k(a)}^2$.
\end{theorem}
} 
\newcommand{\gammaR}[1]{\gamma^{(#1)}}

\section{Verifying noncorrelation}\label{sec:Implementation}
\JK{change title}
{\color{good} 
 We now discuss the practical details of how one can check if a given \pc{} sequence is noncorrelated. 
We begin by setting up the notation and adapting the general results from previous sections to the situation at hand; this is done in subsections \ref{ssec:Imp-Setup} and \ref{ssec:Imp-Recursive}. Then, in subsections \ref{ssec:Imp-Small} and \ref{ssec:Imp-Basis} we discuss how the relevant computations can be performed. Finally, in subsection \ref{ssec:Imp-Complex} we discuss the complexity of the resulting algorithm, which finishes the proof of Theorem \ref{thm:B}. Implementation of this algorithm allows us to verify Theorem \ref{thm:A} by direct computation. 
}

\subsection{Setup}\label{ssec:Imp-Setup}{\color{final}
 Throughout this section, $A \subset \Sigma_k^{\ell}$ denotes an admissible set and $a \colon \NN_0 \to \{+1,-1\}$ denotes the corresponding \pc{} sequence:
\[ a(n) = a_A(n) = (-1)^{\#(A,n)} \qquad (n \in \NN_0).\]
We also introduce the sequence $f \colon \NN_0 \to \RR$ given by
\[
	f := \ind_{\NN} \cdot \gamma_{a}.
\]

Our task amounts to verifying that $f$ is well-defined (i.e., that the limits defining $\gamma_{a}(m)$ exist for all $m \in \NN$) and determining whether it is identically zero. The existence question is easily accounted for (cf.{} \cite[Section 3]{ZhengPengKamae-2018}).
}

\begin{lemma}\label{lem:gamma-exists-PC}\color{final}
	For each $b,c \in \cN_k(a)$ and $m \in \NN_0$, the coefficient $\gamma_{b,c}(m)$ exists.
\end{lemma}
\begin{proof}
By Lemma \ref{lem:char-of-PC}, all sequences in $\cN_k(a)$ are products of $a$ and $k^{\ell-1}$-periodic sequences. Hence, there is a $k^{\ell-1}$-periodic sequence $h$ such that $b(n)c(n) = a(n)^2 h(n) = h(n)$ for all $n \in \NN_0$, and consequently
\[
	\gamma_{b,c}(0) = \lim_{N \to \infty}  \frac{1}{N} \sum_{n=0}^{N-1} h(n) = \frac{1}{k^{\ell-1}} \sum_{n=0}^{k^{\ell-1}-1} h(n) 
\] 
exists. Existence of $\gamma_{b,c}(m)$ for $m \in \NN$ now follows from Lemma \ref{lem:gamma-exists}.
\end{proof}

{\color{final}
Recall that $f = \ind_{\NN} \cdot \gammalog_a$ is $k$-regular by Theorem \ref{cor:gamma-is-reg}.
In principle, in order to decide if $f$ is identically zero, it is now enough to follow the arguments in Section \ref{sec:Correlation} to describe the structure of the $k$-kernel of $f$  and then apply Lemma \ref{lem:zero-criterion}. In practice, we essentially follow this route, but we also take advantage of the fact that $f$ is a $k$-regular sequence of a rather specific form.
}

\subsection{Recursive relations}\label{ssec:Imp-Recursive}{\color{final}
As a first step towards describing the recursive relations that define $f$, we introduce a set that spans $\cN_k(f)$, in analogy to Proposition \ref{prop:gamma-is-reg}. It will be convenient to introduce the restricted averages
\begin{equation}\label{eq:def-of-gamma(r,j)}
	\gammaR{r}(m) := \lim_{N \to \infty} \frac{k^\ell}{\log N} \sum_{n=0}^{N-1}\frac{1}{n+1} a(n)a(n+m) \ind_{k^\ell \NN_0+r}(n).
\end{equation} 
Note that these averages are well-defined thanks to Theorem \ref{thm:log-mean-exists}. Additionally, it follows from Lemma \ref{lem:char-of-PC} and Lemma \ref{lem:gamma-exists-PC} that the logarithmic averages can be replaced with unweighted averages:
\begin{equation}
	\gammaR{r}(m) = \lim_{N \to \infty} \frac{k^\ell}{N} \sum_{n=0}^{N-1} a(n)a(n+m) \ind_{k^\ell \NN_0+r}(n).
\end{equation} 
 As a direct consequence of the relevant definitions, we have
\begin{equation}\label{eq:423:1}
	f = \frac{\ind_{\NN}}{k^\ell} \cdot \sum_{r=0}^{k^{\ell}-1} \gammaR{r} = 
	\frac{1}{k^\ell} \sum_{q=1}^{k^\ell} \sum_{r=0}^{k^{\ell}-1} \ind_{k^\ell \NN_0 + q} \cdot \gammaR{r}.
\end{equation}
}

\begin{proposition}\label{prop:improved-bounds}\color{final}
Each sequence in $\cN_k(f)$ is a linear combination of the sequences $\ind_{k^\ell \NN_0 + q} \cdot S^e \gammaR{r}$, where $e \in \{0,1\}$, $0 \leq r < k^{\ell}$ and $0 \leq q \leq k^{\ell}$. In particular, $\dim \linspan_{\QQ} \cN_k(f) \leq 2k^{\ell}(k^\ell+1)$.
\end{proposition}

{\color{final}
The proof of the above proposition will follow directly once we describe the behaviour of the base sequences $\ind_{k^{\ell} \NN_0 + q}\cdot S^e\gammaR{r}$ under the operators $\Lambda_i$ ($i \in \Sigma_k$).
To simplify this description, it will be convenient to introduce the auxiliary sequence $h \colon \NN_0 \to \{+1,-1\}$, given by
\begin{equation}\label{eq:def-of-b}
	h(n) := a(n)/a(\floor{n/k}).
\end{equation}
}

The following basic fact is analogous to \cite[Lemma 2.1]{ZhengPengKamae-2018}.	

\begin{lemma}\label{lem:b-is-periodic}\color{final}
	The sequence $h$ given by \eqref{eq:def-of-b} is $k^{\ell}$-periodic.
\end{lemma}\color{final}
\begin{proof}
Follows immediately from Lemma \ref{lem:char-of-PC}.
\end{proof}

\begin{lemma}\label{lem:rec-PC}\color{final}
	Let $e \in \{0,1\}$, $i \in \Sigma_k$, $0 \leq q \leq k^{\ell}$ and $0 \leq r < k^\ell$. If $i \neq q \bmod{k}$ then $\Lambda_i \ind_{k^\ell \NN_0 + q} = 0$. If $i = q \bmod{k}$ then 
	\begin{equation}\label{eq:164:1}
		\Lambda_i\bra{ \ind_{k^\ell \NN_0+q} \cdot S^{e} \gammaR{r} } = 
		\frac{h(r)h(r+q+e)}{k} \sum_{q'} \sum_{r'} \ind_{k^\ell \NN_0 + q'} \cdot S^{e'}\gammaR{r'},
	\end{equation}
	where the value of $e$ and the ranges of the summations are given by
	\[ e' := \floor{ \frac{i+e+ (r \bmod k)}{k}},\quad q' \in k^{\ell-1} \Sigma_k + \floor{q/k},\quad r' \in k^{\ell-1} \Sigma_k + \floor{r/k}.\]
\end{lemma}
\begin{proof}
	The case $e = 0$ follows by a standard adaptation of the proof of Lemma \ref{lem:gammalog-recurrence}. Then, the case $e = 1$ is derived using Lemma \ref{lem:ops-compose}.
\end{proof}

\subsection{Small shifts}\label{ssec:Imp-Small}{\label{ssec:SmallShifts}{\color{good}
	Bearing in mind that we hope to apply Lemma \ref{lem:zero-criterion}, we need to be able to compute the values $S^e \gammaR{r}(0) = \gammaR{r}(e)$ for $e \in \{0,1\}$ and $0 \leq r \leq k^\ell$. This can, in principle, be accomplished by straightforward adaptations of the arguments in Lemma \ref{lem:gamma-exists} and Lemma \ref{lem:gamma-exists-PC}. Here, we discuss the practical details of how the computations are performed. Recall that $\gammaR{r}(0) = 1$, so we only need to compute $\gammaR{r}(1)$.
	
	For $0 \leq r < k^\ell$, let $\nu = \nu(r)$ denote the first position where a digit distinct from $k-1$ appears in the base-$k$ expansion $(r)_k$; if $r = k^\alpha-1$ for some $\alpha \geq 0$ then $\nu = \alpha$. We consider $r$ in nondecreasing order with respect to $\nu(r)$. We have three ranges to consider: $r = 0$, $1 \leq r < \ell$ and $r = \ell.$
	
	If $\nu(r) = 0$ then it follows from Lemma \ref{lem:rec-PC} that
	\begin{equation}\label{eq:104:1}
		\gammaR{r}(1) = \frac{h(r)h(r+1)}{k} \sum_{r'} \gammaR{r'}(0) = h(r)h(r+1);
	\end{equation} 
	here and elsewhere, the summation over $r'$ runs through $r' \in k^{\ell-1} \Sigma_k + \floor{r/k}$. Since we can readily compute $h(r)$ and $h(r+1)$, we can compute $\gammaR{r}(1)$.
	
	If $1 \leq \nu(r) < \ell$ then another application of Lemma \ref{lem:rec-PC} yields
	\begin{equation}\label{eq:104:2}
		\gammaR{r}(1) = \frac{h(r)h(r+1)}{k} \sum_{r'} \gammaR{r'}(1).
	\end{equation} 
	For all $r'$ appearing in the above sum we have $\nu(r') = \nu(r)-1$, and hence $\gammaR{r'}(1)$ has been previously computed. Hence, again, we can directly compute $\gammaR{r}(1)$.
	
	Finally, if $\nu = \ell$ (meaning that $r = k^\ell-1$) then \eqref{eq:104:2} continues to hold, and we have $\nu(r') = \ell -1$ for all summands on the right-hand-side except for the one corresponding to $r' = r$. Hence, we can compute $\gammaR{r}$ as
	\begin{equation}\label{eq:104:3}
		\gamma_{k^\ell-1}(1) = \frac{1}{k h(r)h(r+1)-1} \sum_{i=0}^{k-2} \gamma_{k^{\ell-2}(ki+1)-1}(1).
	\end{equation} 	
}

{\color{good}
\subsection{Basis construction}\label{ssec:Imp-Basis}
Recall that our general strategy calls for a construction of a spanning set of $\linspan_{\QQ} \cN_k(f)$.
For technical reasons, it appears to be slightly more convenient and efficient to instead work with the potentially larger space
\[
	\cM := \linspan_{\QQ}\set{ 1_{k^{\ell} \NN_0 + q} \cdot g }{ g \in \cN_k(f),\ 0 \leq q \leq k^{\ell}}.
\]
It remains true that $f = 0$ if and only if $\cM = \{0\}$, and that $\cM$ is closed under $\Lambda_i$ for all $i \in \Sigma_k$. Additionally, $\cM$ admits a decomposition
\[
	\cM = \bigoplus_{q=0}^{k^\ell} \cM_q, \qquad \cM_q := \eta_q \cdot \linspan_{\QQ} \cN_k(f),
\]
where the sequences $\eta_q$ are given by
\[ 
\eta_0 = \ind_{\{0\}},\qquad \eta_q = \ind_{k^\ell \NN_0 + q} \text{ for } 1 \leq q \leq k^{\ell}.  
\]
By Lemma \ref{lem:zero-criterion}, to show that $\cM = \{0\}$ it suffices to verify that $g(0) = 0$ for each $g \in \cM$, which is trivially satisfied for $g \in \cM_q$ for all $1 \leq q \leq k^{\ell}$.

We proceed to construct a list of sequences $f_1,f_2,\dots \in \cM$ which spans $\cM$. Additionally, we ensure that for each $t \geq 1$, the sequence $f_t$ belongs to $\cM_{q_t}$ for some $0 \leq q_t \leq k^\ell$ and we keep track the value of $q_t$. By Proposition \ref{prop:improved-bounds}, each $f_t$ has a decomposition
\begin{equation}\label{eq:105:4}
	f_t = \sum_{r=0}^{k^\ell-1} \sum_{e=0}^1 w_{r,e}^{(t)} \eta_t S^e\gammaR{r},
\end{equation}
for some coefficients $w_{r,e}^{(t)}$, which we also keep track of. While we cannot ensure that $f_1,f_2,\dots$ are linearly independent (in fact, we are primarily interested in the case when $f_1=f_2=\dots=0$), we will ensure that  for each $1 \leq q \leq k^\ell$, the (multi-)set of coefficient vectors $\set{ w^{(t)} }{ q_t = q} \subset \RR^{2k^\ell}$ is linearly independent.

We start by setting for $1 \leq t \leq k^\ell$,
\begin{equation}\label{eq:105:1}
f_t = \ind_{k^\ell \NN_0 + t} \cdot \sum_{r=0}^{k^\ell} \gammaR{r},\quad q_t = t,
\end{equation}
and accordingly $w^{(t)}_{r,e} = \ind_{\{0\}}(e)$ ($0 \leq r < k^\ell$, $e \in \{0,1\}$).

Suppose next that at a certain stage we have constructed $f_1,f_2,\dots,f_v$ and that for all $1 \leq t \leq u$ we have ensured that $\Lambda_i f_t \in \linspan_{\QQ}\{f_1,f_2,\dots,f_v\}$ for all $i \in \Sigma_k$. (Initially, $v = k^\ell$ and $u = 0$.) If $u = v$ then $\linspan_{\QQ}\{f_1,f_2,\dots,f_v\}$ is a subset of $\cM$ that is closed under $\Lambda_i$ ($i \in \Sigma_k$) and under multiplication by $\ind_{k^\ell \NN_0 + q}$ ($0 \leq q \leq k^\ell$), hence $\linspan_{\QQ}\{f_1,f_2,\dots,f_v\} = \cM$ and the construction is complete.

Let us next consider the case when $u < v$. Put $q = q_{u+1}$, $g = f_{u+1}$ and $w = w^{(u+1)}$. Recall that the only value of $i$ for which $\Lambda_i g$ could be non-zero is $i = q \bmod{k}$. If $q = 0$ then $g = g(0) \ind_{\{0\}}$. Hence, either $g(0) \neq 0$, in which case $a$ is not noncorrelated and we are done; or $g(0) = 0$, in which case $g = 0$ and so $\Lambda_i g = 0$ as well.
Suppose now that $1 \leq q \leq k^\ell$. Applying Lemma \ref{lem:rec-PC}, we obtain a representation of $\Lambda_i g$ in the form
\begin{equation}\label{eq:105:2}
	\Lambda_i g = \sum_{q'} \sum_{r'} \sum_{e'} w'_{q',r',e'} \ind_{k^\ell \NN_0 + q'} \cdot S^{e'}\gammaR{r'}, 
\end{equation}
where the ranges of summation are given by $0 \leq q' \leq k^{\ell}$, $0 \leq r' < k^{\ell}$ and $0 \leq e' \leq 1$, and the coefficients $w'$ are given by explicit formulae coming from \eqref{eq:164:1}. Bearing in mind that $\ind_{k^\ell \NN_0} = \ind_{k^\ell \NN} + \ind_{\{0\}}$, we find the decomposition
\begin{equation}\label{eq:164:8}
	\Lambda_i g = \sum_{q'} g_{q'}',\qquad g_{q'}' = \sum_{r'} \sum_{e'} w''_{q',r',e'} \eta_{q'} S^{e'}\gammaR{r'}, 
\end{equation}
where the coefficients $w''_{q',r',e'}$ are given by:
\begin{align}\label{eq:105:3}
w''_{q',r',e'} = w'_{q',r',e'} \text{ if } q' \neq k^\ell,\qquad w''_{k^{\ell},r',e'} = w'_{k^{\ell},r',e'} + w'_{0,r',e'}.
\end{align}
For each $q'$, we append $g'_{q'}$ to the list $f_1,f_2,\dots,f_v$ if (and only if) 
\begin{equation}\label{eq:164:9}
	\bra{ w''_{q',r',e'} }_{r',e'} \not \in \linspan_{\QQ} \set{ \bra{ w^{(t)}_{r,e} }_{r,e} }{1 \leq t \leq v,\ q_t = q'}. 
\end{equation}
If \eqref{eq:164:9} holds then we also record $g'_{q'} \in \cM_{q'}$ (that is, we append $q'$ to the list $q_1,q_2,\dots, q_v$) and that the decomposition of $g'_{q'}$ as the sum of basis sequences is given by \eqref{eq:164:8} (what is, we append $w''_{q'}$ to the list $w^{(1)},w^{(2)},\dots, w^{(v)}$. Each time a new sequence is added, $v$ increases by $1$ and after all $q'$ have been processed, $u$ increases by $1$.

The linear independence condition \eqref{eq:164:9} ensures that for each $1 \leq q \leq k^\ell$, there are at most $2k^\ell$ values of $t$ with $q_t = q$, and hence the construction needs to terminate after a bounded number of steps. As the result, we either find, for some $t \geq 1$, a sequence $f_t \in \cM$ with $f_t(0) \neq 0$ (in which case $a$ is not noncorrelated) or we construct a finite list of sequences $f_1,f_2,\dots,f_N \in \cM$ that spans $\cM$ and satisfies $f_t(0) = 0$ for all $1 \leq t \leq N$ (in which case $a$ is noncorrelated). In either case, we are able to determine whether $a$ is noncorrelated. 
}

\subsection{Complexity}\label{ssec:Imp-Complex}
We now provide quantitative estimates for the amount of computational power needed to verify if the \pc{} sequence $a$ is noncorrelated using the method described above. Throughout, we treat $k$ as fixed, and hence are interested in the regime $\ell \to \infty$. It will be convenient to introduce, for a function $F \colon \NN \to \RR_{>0}$, the shorthand $\tilde O( F(\ell))$ to denote $O(\ell^{O(1)} F(\ell))$. Thus, for instance, addition or multiplication of two integers of size $O(k^\ell)$ can be performed using $\tilde O(1)$ operations.

At several points, we need to compute the values of $a(n)$ where $n = O(k^\ell)$. For a word $w \in \Sigma_k^*$ with length $\abs{w} \leq \ell$, computing $\# (n,w)$ directly from the definition requires $\tilde O(1)$ operations. Since $\abs{A} \leq k^\ell$, the values $\#(n,A)$ and $a(n)$ can be computed in time $\tilde O(k^\ell)$ . Consequently, we can also compute $h(n)$ in time $\tilde O(k^\ell)$.

Following the steps in subsection \ref{ssec:Imp-Small}, we compute $\gamma^{(r)}(1)$ for all $0 \leq r < k^\ell$. It takes $\tilde O(k^\ell)$ operations to write the values of $r$ ($0 \leq r < k^\ell$) in an order consistent with $\nu(r)$. Note that each of the formulae \eqref{eq:104:1}, \eqref{eq:104:2}, \eqref{eq:104:3} produces the corresponding value of $\gamma^{(r)}(1)$ using $\tilde O(1)$ arithmetic operations on rational numbers. One can also check by a simple inductive argument that all denominators and numerators that appear in these computations are bounded by $O(k^\ell)$, and hence each arithmetic operation takes only $\tilde O(1)$ basic operations. We also note that all the denominators take the form $(k \pm 1) k^\alpha$.

We next proceed to the computation of the sequences $f_t$ ($t=1,2,3,\dots$) in subsection \ref{ssec:Imp-Basis}. Strictly speaking, we compute the sequence $w^{(t)}$, which uniquely determine $f_t$ via \eqref{eq:105:4}, and the auxiliary sequence $q_t$. For $t \leq k^\ell$, the explicit formula \eqref{eq:105:1} allows us to compute $w^{(t)}$ and $q_t$ with $\tilde O(k^{2\ell})$ operations (note that $w^{(t)} = \big(w^{(t)}_{r,e}\big)_{r,e}$ has $k^{2\ell}$ entries, so this is the least number of operations possible). 

Let us now consider the amount of computation required to compute $f_t$ for $t > k^\ell$. 
Consider any $u,v$, as in the iterative procedure in second half of subsection \ref{ssec:Imp-Basis}. We note that the application of Lemma \ref{lem:rec-PC} used to compute $w'$ in \eqref{eq:105:2} requires no more than $\tilde O(k^{3\ell})$ arithmetic operations (for each of $O(k^\ell)$ summands in the decomposition of $g$, we substitute a sum of size $O(k^{2\ell})$). Once $w'$ is computed, it only takes $\tilde O(k^{2\ell})$ operations to compute $w''$. 
 Then, for each of $O(k^\ell)$ values of $q'$, in order to verify if $g'_{q'}$ should be appended to the list $f_1,f_2,\dots$, we need to verify if the corresponding vector of coefficients belongs to a certain linear subspace of $\RR^{2k^\ell}$, see \eqref{eq:164:9}. Keeping track of how much the complexity increases in each step of the construction, we see that for each $t > k^{\ell}$, the entries of $w^{(t)}$ are rational numbers whose numerators are $\tilde O(k^{3t})$, and whose denominators are $O(k^{t})$ and divide $(k^2-1)k^\alpha$ for some integer $\alpha$. Thus, in \eqref{eq:164:9} we may scale all of the relevant vectors by a factor of $(k^2-1)k^{O(u)}$, leaving us with the task of verifying if an integer-valued vector belongs to the span of other integer-valued vectors. The latter task is well-known to have polynomial complexity (with respect to dimensions and lengths of representations of entries), see e.g.{} \cite[Chpt.{} 16]{BCS-book}. Hence, for each $q'$ in order to decide if $g'_{q'}$ should appended, we perform $\tilde O(k^{O(\ell)}) = k^{O(\ell)}$ operations. Consequently, the number of operations needed to process the step corresponding to the index $u$ is $k^{O(\ell)}$.
 
Because of the linear independence conditions discussed at the end of subsection \ref{ssec:Imp-Basis}, the total number of the sequences $f_1,f_2,\dots$ we construct is at most $2k^{2\ell+1}$. It follows that in total, we perform at most $k^{O(\ell)}$ operations.

\color{good}
\section{\Sinv{} sequences}

We now turn to the classification of \sinv{} \pc{} sequences. Throughout, let $A \subset \Sigma_k^*$ be a set of patterns with no leading or trailing zeros, and let $a = a_A$ be the corresponding \pc{} sequence. We also retain the notation from Section \ref{sec:Implementation}, specifically the coefficients $\gamma_r$ defined in \eqref{eq:def-of-gamma(r,j)}. We let $\ell = \max_{v \in A} \abs{v}$ denote the length of $a$, and we assume that $\ell \geq 2$.

The following condition turns out to be closely connected to the question of whether $a$ is noncorrelated: 
\begin{equation}\label{eq:saturated}\tag{$\dagger$}
  \abs{ A(ui_0)^{-1} \oplus A(ui_1)^{-1}} = \frac{k}{2} \text{ for each } u \in \Sigma_k^{\ell-2} \text{ and } i_0,i_1 \in \Sigma_k \text{ with } i_0 \neq i_1.
\end{equation}
Above, using the standard notation from semigroup theory, for a word $u \in \Sigma_k^*$ and a set $X \subset \Sigma_k^*$, we let $Xu^{-1} := \set{v \in \Sigma_k^*}{vu \in X}$.

\begin{remark}\label{rmk:saturated-2}
	The condition \eqref{eq:saturated} can be stated in simpler terms when $k = 2$. Then, necessarily, $\{i_0,i_1\} = \{\mathtt{0},\mathtt{1}\}$ and since $A$ has no trailing zeros, $A\mathtt{0}^{-1} = \emptyset$. Hence, \eqref{eq:saturated} says that $\abs{ A(u\mathtt{1})^{-1} } = 1$ for all $u \in \Sigma_k^{\ell-2}$. Because all patterns in $A$ have length $\leq \ell$, $A(u\mathtt{1})^{-1} \subset  \{\mathtt{0},\mathtt{1}\}$; and because $A$ has no leading zeros, $\mathtt{0} \not \in A(u\mathtt{1})^{-1}$. Thus, \eqref{eq:saturated} reduces to the statement that $\mathtt{1} u \mathtt{1} \in A$ for all $u \in \Sigma_k^{\ell-2}$, that is, $\mathtt{1}\Sigma_2^{\ell-2}\mathtt{1} \subset A$. This is precisely the assumption that appears in Theorem \ref{thm:C}.
\end{remark}

\begin{remark}\label{rmk:saturated-large-k}
	For general $k \geq 2$, it is not a priori clear if there exists a set of patterns $A$ such that \eqref{eq:saturated} holds. Fix $u \in \Sigma_k^{\ell-2}$ and consider the matrix $M = \big(M_{i,j}^{(u)}\big)_{i,j=0}^{k-1}$ where $M_{i,j}^{(u)} = -1$ if $iuj \in A$ and $M_{i,j}^{(u)} = +1$ otherwise. Then \eqref{eq:saturated} says that $M^{\mathrm{T}}M = k I$, where $I$ denotes the identity matrix, meaning that $M$ is a Hadamard matrix. Additionally, $M^{(u)}_{i,j} = +1$ if $i = 0$ or $j = 0$, meaning that $M$ is normalized. Conversely, given any normalized Hadamard matrix $M'$, one can easily reconstruct $A$ so that $M = M'$ for each choice of $u \in \Sigma_k^{\ell-2}$. Thus, it is possible to satisfy the condition \eqref{eq:saturated} if and only if there is at least one Hadamard matrix of dimension $k$.
	
	 The question of existence of Hadamard matrices of a given dimension has long been investigated. They are easily constructed when $k$ is a power of $2$ through a tensor-power construction. More generally, given Hadamard matrices of dimensions $k$ and $k'$ one can construct a Hadamard matrix of dimension $k \cdot k'$.  It is conjectured that Hadamard matrices exist for $k=1,2$ and all $k$ divisible by $4$. So far, this has been confirmed for $k < 668$. See e.g.{} \cite[Chpt.{} V]{ChapmanHall-book} for further discussion.
\end{remark}

The main goal of this section is to prove a slightly more general variant of Theorem \ref{thm:C}. The second part of this theorem asserts that if $a$ is noncorrelated, $k = 2$ and $\ell \leq 5$ then \eqref{eq:saturated} holds. This is verified by exhaustive search\footnote{Code available from the author.}, using the methods developed in Section \ref{sec:Implementation}. The remaining part of Theorem \ref{thm:C} follows from the following result, whose proof will occupy the remainder of this section.

\begin{proposition}\label{prop:sat->NC}
	Suppose that \eqref{eq:saturated} holds. Then the sequence $a$ is noncorrelated. 
\end{proposition}

From this point onwards, assume that \eqref{eq:saturated} holds. Proceeding along similar lines as in Lemma \ref{lem:gamma-exists} (or Section \ref{ssec:SmallShifts}), we will compute $\gamma_r(m)$ for small values of $m \in \NN_0$ ($0 \leq r < k^\ell$).
The following lemma is the main consequence of \eqref{eq:saturated} that we use.

\begin{lemma}\label{lem:sat-cancel}
	Let $u \in \Sigma_k^{\ell-2}$ and $j_0,j_1 \in \Sigma_k$, $j_0 \neq j_1$. Then
	\begin{equation}\label{eq:765:1}
		\sum_{i = 0}^{k-1} a\bra{[iuj_0]_k}a\bra{[iuj_1]_k} = 0. 
	\end{equation}
\end{lemma}
\begin{proof}
	Multiplying by $ a\bra{[uj_0]_k}a\bra{[uj_1]_k}$, we see that \eqref{eq:765:1} is equivalent to
	\begin{equation}\label{eq:765:1b}
		\sum_{i = 0}^{k-1} a\bra{[uj_0]_k}a\bra{[uj_1]_k} a\bra{[iuj_0]_k}a\bra{[iuj_1]_k}  = 0. 
	\end{equation}
	Each pattern $v$ in $A$ of length $< \ell$ and each $i \in \Sigma_k$, considering the different positions where $v$ can appear, one can check that 
	\[
		\#(v,uj_0) + \#(v,iuj_1) = \#(v,uj_1) + \#(v,iuj_0).
	\]
	Conversely, if $v \in A$ and $\abs{v} = \ell$ then 
	\[
		\#(v,uj_0) = \#(v,uj_1) = 0, \qquad 		
	\]
	since $\abs{uj_0},\abs{uj_1} < \ell$, and for each $i \in \Sigma_k$
	\[
	\#(v,iuj_0) + \#(v,iuj_1) =
	\begin{cases}
		1 &\text{ if } v \in \{iuj_0, iuj_1\},\\
		0 &\text{ otherwise.}
	\end{cases}
	\]
	Substituting the above identities into the sum on the left-hand side of \eqref{eq:765:1} and applying \eqref{eq:saturated} we conclude that
	\begin{align*}
			\sum_{i = 0}^{k-1} a\bra{[iuj_0]_k}a\bra{[iuj_1]_k} 
			&= \sum_{i = 0}^{k-1} (-1)^{\# \{iuj_0, iuj_1\} \cap A} 
			\\ & = k-2\abs{A(uj_0)^{-1} \oplus A(uj_1)^{-1}} =0.\qedhere
	\end{align*}
\end{proof}

\begin{lemma}\label{lem:sat-gamma_r(i)}
	Let $0 \leq r < k^\ell$ and $m \geq 0$. Put $j = r \bmod{k}$. Then
	\[
		\gamma_{r}(m) =
		\begin{cases}
a(r)a(r+m) & \text{ if } j+m < k \text{ and } m \neq 0, \\
		0 & \text{ otherwise.}
		\end{cases}
	\]
\end{lemma}
\begin{proof}
Let us write $m = km'+i$ with $m' \geq 0$ and $i \in \Sigma_k$. Then by Lemma \ref{lem:rec-PC} (or, equivalently, by Lemma \ref{lem:gammalog-recurrence}) we have
	\begin{equation}\label{eq:691:1}
		\gamma_r(m) = \gamma_r(km'+i) = \frac{h(r)h(r+km'+i)}{k}\sum_{r'} \gamma_{r'}(m'+e'),
	\end{equation}
	where as usual $r' \in k^{\ell-1} \Sigma_k + \floor{r/k}$ and $e' = \floor{(i+j)/k} \in \{0,1\}$. We consider several different cases.
	
\textit{Case 0:} $m = 0$. It follows directly from the definition of $\gamma_r$ that 
\[\gamma_{r}(0) = 1 = a(r)^2 = a(r)a(r+m).\] 

\textit{Case 1:} $m \neq 0$ and $j+m < k$. Applying \eqref{eq:691:1} and noticing that $i = m$, $m'=0$, and $e' = 0$, we obtain
	\begin{equation}\label{eq:691:2}
		\gamma_r(m) = h(r)h(r+m) = a(r)a(r+m),
	\end{equation}
where the second equality holds because $\floor{r/k} = \floor{(r+m)/k}$.

In all of the remaining cases, we will show that $\gamma_r(m) = 0$. We start with the simplest situation where $e' = 1$.

\textit{Case 2:} $m=1$ and $j+m \geq k$, meaning that $j = k-1$. Let $\nu(r)$ denote the first position where a digit distinct from $k-1$ appears in the expansion of $r$, allowing $\nu(r) = \alpha$ if $r = k^\alpha-1$. By \eqref{eq:691:1},
	\begin{equation}\label{eq:691:3}
		\gamma_r(1) = \pm \frac{1}{k}\sum_{r'} \gamma_{r'}(1).
	\end{equation}
If $\nu(r) = 1$ then from the previously considered cases and Lemma \ref{lem:sat-cancel} it follows that
\[
		\gamma_r(1) = \pm \frac{1}{k}\sum_{r'} a(r')a(r'+1) = 0.
\]
If $1 < \nu(r) < \ell$ then $\nu(r') = \nu(r) - 1$ for all $r'$ that enter the sum \eqref{eq:691:3}. Hence, reasoning by induction on $\nu(r)$ we conclude that $\gamma_r(1) = 0$. Finally, if $\nu(r) = \ell$ then $r = k^\ell-1$, and $\nu(r') = \ell - 1$ for all $r'$ that appear in the sum \eqref{eq:691:3} except for $r' = r$. It follows that
	\[
	\gamma_{k^{\ell}-1}(1) = \pm \frac{1}{k} \gamma_{k^{\ell}-1}(1),
	\]
which is only possible if $\gamma_{k^{\ell}-1}(1) = 0$.

\textit{Case 3:} $2 \leq m < k$ and $j + m \geq k$.  By \eqref{eq:691:1} and Case 2,
	\begin{equation}\label{eq:691:31}
		\gamma_r(m) = \pm \frac{1}{k}\sum_{r'} \gamma_{r'}(1) = 0.
	\end{equation}

\textit{Case 4:} $k \leq m < k^2$.  By \eqref{eq:691:1},
	\begin{equation}\label{eq:691:32}
		\gamma_r(m) = \pm \frac{1}{k}\sum_{r'} \gamma_{r'}(m'+e') = 0.
	\end{equation}
	Let $j' := \floor{r/k} \bmod r$ and $i' := m'+e'$. Note that $r' \bmod{k} = j'$ for all $r'$ in the sum in \eqref{eq:691:32}, where we are using the fact that $\ell \geq 2$. We have several subcases to consider. If $j'+i' < k$ then
\[
		\gamma_r(m) = \pm \frac{1}{k}\sum_{r'} a(r')a(r'+i') = 0
\]
	by Cases 0 and 1 and Lemma \ref{lem:sat-cancel}. If $j'+i' \geq k$ while $i' < k$ (i.e.{} $m' \neq k-1$ or $e' \neq 1$) then $\gamma_{r'}(m'+e') = 0$ for all $r'$ by Cases 2 and 3, and consequently also $\gamma_r(m) = 0$. Finally, if $j'+i' = k$ (i.e.{} $m'=k-1$ and $e' = 1$) then 
\[
	\gamma_r(m) = \pm \frac{1}{k}\sum_{r'} \gamma_{r'}(k\cdot 1 + 0) = 0
\]
by the previously considered subcases.

\textit{Case 5:} $m \geq k^2$.  We reason by induction on $m$. By \eqref{eq:691:1} and the inductive assumption,
	\begin{equation}\label{eq:691:33}
		\gamma_r(m) = \pm \frac{1}{k}\sum_{r'} \gamma_{r'}(m'+e') = 0
	\end{equation}
	since $k \leq m'+e' < m$.
\end{proof}

Now that we have computed the values of the coefficients $\gamma_r(m)$, the remainder of the argument is straightforward.

\begin{proof}[Proof of Proposition \ref{prop:sat->NC}]
We need to show that 
\[
	\gamma(m) = \frac{1}{k^\ell}\sum_{r=0}^{k^\ell-1} \gamma_r(m) = 0
\] 
for all $m \geq 1$. If $m \geq k$ there is nothing to prove since $\gamma_r(m) = 0$. Suppose now that $1 \leq m <k$. We may write arbitrary $0 \leq r < k^\ell-1$ in the form $r = k^{\ell-1}i + ks + j$ where $i,j \in \Sigma_k$ and $0 \leq s < k^{\ell-2}$. Then, $\gamma_r(m) = 0$ if $j+m \geq k$ and $\gamma_r(m) = a(r)a(r+m)$ otherwise. It follows that
\[
	\gamma(m) = \sum_{j=0}^{k-m-1} \sum_{s = 0}^{k^{\ell-1}-1} \sum_{i=0}^{k-1} a\bra{k^{\ell-1}i + ks+j}a\bra{k^{\ell-1}i +ks+j+m} = 0,
\]
where the inner-most sum vanishes by Lemma \ref{lem:sat-cancel}.
\end{proof}	
	
\begin{remark}\label{rmk:final}
	Let $a' \colon \NN_0 \to \{+1,-1\}$ be a sequence such that $a'/a$ is $k^{\ell-1}$-periodic. Then $a'$ is \pc{} by Lemma \ref{lem:PC-representation}. Defining $\gamma'$ and $\gamma_r'$ in analogy to $\gamma$ and $\gamma_r$, with $a'$ in place of $a$, by a direct computation we show for all $m \geq 0$ and $0 \leq r < k^\ell$ that 
\begin{equation}\label{eq:440:1}
		\gamma_r'(m) = \frac{a'(r)a'(r+m)}{a(r)a(r+m)} \gamma_r(m) = \pm \gamma_r(m).
\end{equation}
It follows that $\gamma_r'(m) = 0$ for all $m \geq k$. In particular, $\gamma'(m) = 0$ for all $m \geq k$.

We check by exhaustive search that all noncorrelated binary \pc{} sequences of length $\leq 4$ can arise as $a'$ in the construction outlined above. It seems plausible that the same holds for all lengths. If this is the case, and if Conjecture \ref{conj:NC+inv->sat} holds true, then the task of verifying if a given binary \pc{} sequence $b'$ is noncorrelated can be split into two independent steps: First, check if the \sinv{} sequence $b$ obtained from $b'$ in Lemma \ref{lem:PC-rep-invariant} satisfies \eqref{eq:saturated}; if not then $b'$ is not noncorrelated\footnote{For the sake of simplicity, we work under the additional assumption that $b$ and $b'$ have equal lengths, which is not true in general.}. Second, check if the $\pm$ signs in (the analogue of) \eqref{eq:440:1} align in a way that ensures $\gamma_{b'}(1) = 0$. While the condition from the first step is quite conceptual, it appears that the second step relies mostly on arithmetic coincidence. This would provide an intuitive explanation for why the results in the \sinv{} case are considerably more concise.

\end{remark}

\bibliographystyle{alphaabbr}
\bibliography{bibliography}

\end{document}